\documentclass[sn-mathphys-num]{sn-jnl}

\usepackage{color}

\usepackage{graphicx}%
\usepackage{multirow}%
\usepackage{amsmath,amssymb,amsfonts}%
\usepackage{amsthm}%
\usepackage{mathrsfs}%
\usepackage[title]{appendix}%
\usepackage{xcolor}%
\usepackage{textcomp}%
\usepackage{manyfoot}%
\usepackage{booktabs}%
\usepackage{algorithm}%
\usepackage{algorithmicx}%
\usepackage{algpseudocode}%
\usepackage{listings}%

\newtheorem{thm}{Theorem}[section]
\newtheorem{cor}[thm]{Corollary}
\newtheorem{lem}[thm]{Lemma}
\newtheorem{prp}[thm]{Proposition}

\theoremstyle{definition}

\newtheorem{rem}[thm]{Remark}

\newcommand{\scr}[1]{\mathscr #1}

\numberwithin{equation}{section} \theoremstyle{remark}

\renewcommand{\d}{\hbox{d}}

\def\R{\mathbb R}
  
\def\N{\mathbb N}

\def\<{\langle} \def\>{\rangle}  
    \def\E{\mathbb E}
\def\d{\text{\rm{d}}}   
 
\def\RR{\mathscr R}
  
\def\beq{\begin{equation}}

\def\F{\scr F}
\def\e{\text{\rm{e}}}

\def\P{\mathbb P}

\renewcommand{\bar}{\overline}
\renewcommand{\tilde}{\widetilde}

\begin{document}

\title[Rivers under noise]{Rivers under Additive Noise}

\author[1]{\fnm{Michael} \sur{Scheutzow}}\email{ms@math.tu-berlin.de}

\author*[2]{\fnm{Michael} \sur{Grinfeld}}\email{m.grinfeld@strath,ac.uk}


\affil[1]{\orgname{Technische Universit\"at Berlin},
  \orgaddress{\street{MA 7-5, Strasse des 17~Juni 136},
    \city{Berlin}}, \postcode{10623},
  \country{Germany}}

\affil*[2]{\orgdiv{Department of Mathematics and Statistics},
  \orgname{University of Strathclyde}, \street{26 Richmond Street},
  \city{Glasgow},
    \postcode{G1 1XH}, \country{UK}}

\abstract{We consider the deterministic and stochastic versions of
      a first order non-auto\-no\-mo\-us differential equation which
      allows us to discuss the persistence of rivers (``fleuves'')
      under additive noise.}



\maketitle

\section{Introduction}

This paper uses a simple example in the pursuit of understanding the
stochastic counterpart of the river (``fleuve'') phenomenon. We
believe that this example contains all the important ingredients of the
general theory. 

Rivers, a remarkable organising feature of phase portraits of
polynomial non-autonomous first order ODEs were discovered in the 1980s
by nonstandard analysts (see \cite{Blais} for the definitions and
additional references, and \cite{Diener} for connections between
rivers and centre manifolds).  Clearly, an interesting question is to
try to extend the notions introduced in the context of rivers to the
case of stochastic differential equations (SDEs).

The present paper is devoted to an analysis of a simple example of the
stochastic counterpart of a polynomial non-autonomous ordinary
differential equation which exhibits rivers in which a trichotomy of
asymptotic fates of solutions occurs. The structure of the paper is as
follows. In Section~\ref{det} we discuss the structure of solutions of
the deterministic version of our equation \eqref{deter}, and exhibit
the asymptotics of the critical solution; this material is standard
and is helped by the existence of an explicit formula for the general
solution. In Section 3, we briefly study the much simpler
(deterministic and random) unstable {\em linear} case in which there
is a repelling river of sub-exponential growth.  This case is easy to
analyze since again there is an explicit formula for the general
solution and for the repelling river. In Section 4, we state our main
result: the non-autonomous quadratic stochastic differential equation
\eqref{SDE}, the additive noise counterpart of \eqref{deter} also
admits a trichotomy: there is a random repelling river with the
property that trajectories starting above it blow up in finite time
while trajectories starting below the repelling river converge to 0
and the river which separates the two regimes has linear growth as
$t \rightarrow \infty$ (as in the deterministic case). Note that,
contrary to the linear case, there is an attracting solution (any
solution which converges to 0 attracts all trajectories starting below
the repelling river).  We formulate our results as four theorems which
we prove in the subsequent sections. Our main tools are well-known
estimates of exit probabilities of diffusions from intervals of the
real line which we quote in the appendix for the reader's convenience.
Finally, in Section~\ref{conc}, we discuss extending other tools
available in the deterministic setting (Wa\.zewski principle,
asymptotic expansions) to the stochastic context.

\section{The deterministic situation} \label{det}

The deterministic equation we start with is
\begin{equation}\label{deter}
  x'= x^2-tx, \quad t>0, \quad x(0)=x_0.
\end{equation}
It is easy to analyse as the general solution of the initial value
problem can be found explicitly,
\begin{equation}\label{expl}
  x(t) = \frac{2x_0\exp\left( - \frac{t^2}{2} \right)}
  {2- x_0\sqrt{2\pi}\, \hbox{erf}\,\left( \frac{t}{\sqrt{2}} \right)}.
\end{equation}

Here $\hbox{erf}(\cdot)$ is the error function,
\[
  \hbox{erf}\, (z) = \frac{2}{\sqrt{\pi}} \int_0^z e^{-t^2} \, \d t,
\]
$\lim_{z \rightarrow \infty} \hbox{erf}\, (z) = 1$, from which we
immediately obtain the following trichotomy of fates of solutions in
positive time : If $x_0> \sqrt{2/\pi}$, solutions blow up in finite
time; if $x_0= \sqrt{2/\pi}$, $x(t) \rightarrow \infty$ as
$t \rightarrow \infty$, and finally if $x_0< \sqrt{2/\pi}$,
$x(t) \rightarrow 0$ as $t \rightarrow \infty$. We call the solution
(\ref{expl}) with $x_0= \sqrt{2/\pi}$ the \emph{critical solution} and
denote it by $x_c(t)$.

As there is symmetry $x\rightarrow -x$, $t\rightarrow -t$, similar
statements can be made as $t \rightarrow  -\infty$; below we restrict
ourselves to positive times. 

Given the explicit form \eqref{expl} of solutions of \eqref{deter},
the asymptotics of the critical solution as $t \rightarrow \infty$ is
easily computed; it is given by
\begin{equation}\label{asymp}
 x_c(t) \sim   t + \frac{1}{t} - \frac{2}{t^3} + O\left(\frac{1}{t^5}\right).
 \end{equation}

 Note that in terms of rivers, $x_c(t)$ provides a repelling river,
 while $x(t) \equiv 0$ and all the positive semi-orbits that converge
 to it, constitute attracting rivers.

The asymptotic expansion of the repelling river in \eqref{asymp} can
be constructed directly from the equation \eqref{deter} without using
\eqref{expl} by following the procedure explained in \cite{Blais}.

\section{The linear case}

Before investigating the behaviour of equation \eqref{deter} with
additive white noise we study the much easier linear stochastic
differential equation
\begin{equation} \label{lin}
\d X(t)=\big(cX(t)+f(t)\big)\,\d t + \sigma \d W(t),\;t \geq 0,
\end{equation}
where $W(t),\,t \geq 0$ is a standard (one-dimensional) Wiener process (also known as Brownian motion), $c>0$, $\sigma \geq 0$ and $f:[0,\infty)\to \R$ is continuous.

The SDE \eqref{lin} with initial condition $X(0)=x$ has the explicit solution
$$
X(t,x)=\e^{ct}\Big( x + \int_0^t \e^{-cs}f(s)\,\d s + \sigma
\int_0^t\e^{-cs}\,\d W(s)\Big).
$$
It follows that all trajectories are exponentially unstable with
$\frac{\d}{\d x}~X(t,x)~=~\e^{ct}$, $t \geq 0,\,x \in \R$.
We now assume that
$$
\int_0^\infty \e^{-cs}|f(s)|\,\d s<\infty.
$$
Below we will be using improper It\^o integrals. For a rigorous
definition of this mathematical object see, e.g., \cite[Definition
25.10]{Klenke}. Denote the unique solution $X(t),\,t \geq 0$ of
\eqref{lin} with the initial condition
$$
X(0)=-\int_0^\infty\e^{-cs}f(s)\,\d s - \sigma \int_0^\infty\e^{-cs}\,\d W(s)
$$
by $\RR(t),\,t \geq 0$ (here and below $\RR$ stands for {\em river}).
Then
\begin{align*}
  X(t, X(0)) & := \RR(t)= -\e^{ct} \int_t^\infty\e^{-cs}f(s)\,\d s -
                \sigma e^{ct}  \int_t^\infty\e^{-cs}\,\d W(s)\\
  & =  -\e^{ct} \int_t^\infty\e^{-cs}f(s)\,\d s -\sigma A(t),
\end{align*}
where
$$
A(t)= \e^{ct}  \int_t^\infty\e^{-cs}\,\d W(s),\;t \geq 0
$$
is a stationary Ornstein-Uhlenbeck process. It follows that
$$
\lim_{t\to \infty}X(t,x)\e^{-ct}=x-\RR(0),
$$
so solutions go to $+\infty$ (respectively, $-\infty$) exponentially
fast with rate $c$ when $x$ is above (respectively below) $\RR(0)$ while
if  $x=\RR(0)$,
\[
  \big|X(t,\RR(0))\big| = \big|\RR(t)\big|=o\big(\e^{ct}\big).
\]
In particular, if in \eqref{lin} we set $f \equiv 0$, then just like in
the deterministic case ($\sigma \equiv 0$) where
the repelling river given  $X(t,x)\equiv 0$,  we have a trichotomy
with random repelling river given by $\RR(t)$.  


\section{Main results}~\label{stoch}

We now study the additive noise stochastic equivalent of \eqref{deter},
namely, the stochastic differential equation
\begin{equation}\label{SDE}
\d X(t)=X(t)\big(X(t)-t\big)\,\d t + \sigma\,\d W(t), \,X(s)=x,\,t \geq s,
\end{equation}
where $W$ is a one-dimensional Wiener process defined on a probability
space $(\Omega,\F,\P)$, $s \geq 0$, $\sigma>0$, and $x \in \R$.  We
will study the long-time behaviour of the unique local solution of
\eqref{SDE} which we denote by $X_{s,t}(x)$, and we will see that
\eqref{SDE} possesses both a random repelling river $\RR$ (which
satisfies $\RR(t)-t \to 0$ as in the deterministic case \eqref{deter})
and an attracting river of trajectories which converge to 0 as
$t \to \infty$ (again as in \eqref{deter}). Just like in the linear
case, $\RR(0)$ is random.  All trajectories starting above $\RR(s)$ at
time $s \geq 0$ blow up to $\infty$ in finite time while all
trajectories starting below $\RR(s)$ converge to 0 as $t \to
\infty$. It will turn out that $\RR(0)$ will be $-\infty$ with
positive probability. In this case, $\RR(t)$ will start becoming
larger than $-\infty$ at some positive random time and then stay
finite for all future times and just like in the nonlinear
deterministic and the linear random cases it will solve the underlying
equation. All these properties (and more) will follow from four
theorems stated below. Further, Proposition \ref{osc} shows that,
unlike the repelling river in the deterministic case in \eqref{deter},
$\RR(t)$ oscillates around $t$ as $t \to \infty$ almost surely.


Before stating our first main result, we collect some 
properties of the following more general model:
\begin{equation}\label{SDEH}
X(t)=x+\int_s^t X(u)\big(X(u)-u\big)\,\d u + H(t)-H(s),\,t \geq s,
\end{equation}
where $H:[0,\infty)\to \R$ is a given (random or deterministic) continuous function satisfying $H(0)=0$, $x \in \R$, and 
$s \geq 0$. Again we denote the unique local solution  of \eqref{SDEH}
by $X_{s,t}(x)$.
We define the blow-up time by
$$
\beta_s(x):=\lim_{n \to\infty}\inf\{t >s: |X_{s,t}(x)|\geq n\}\in
(s,\infty],\,x \in \R.
$$
Note that the local solution  $X_{s,t}(x)$ depends continuously on $(t,x)$. 
Arguing from the contrary from (\ref{SDEH}), we have the
fundamental monotonicity property:

\begin{lem}\label{monot}
  $x < y$ implies $X_{s,t}(x) \leq X_{s,t}(y)$ for $t \geq s$;
  moreover,   $X_{s,t}(x) < X_{s,t}(y)$ when $\beta_s(x)>t$.
\end{lem}


\begin{prp} \label{Xinf}
  For each $x\in \R$ and $v \geq s\geq 0$, we
  have
  $$
\inf_{s\leq t \leq v} X_{s,t}(x)\geq x\wedge 0+\inf_{s \leq t \leq u \leq v} \big(H(u)-H(t)\big).
$$
\end{prp}

\begin{proof}
  Assume that $x \leq 0$ and that $u \in [s,v]$ is such that $X_{s,u}(x)\leq x$. Let $t=t(u)\in [s,u]$ be such that $X_{s,t}(x)=x$ and $X_{s,r}(x)\leq x$ for all $r \in [t,u]$ 
(such a $t$ exists by continuity of $r \mapsto X_{s,r}(x)$ and the fact that $X_{s,s}(x)=x$). Then
\begin{align*}
X_{s,u}(x)&=X_{s,t}(x)+\int_t^u X_{s,r}(x)\big(X_{s,r}(x) -r\big)\,\d r+H(u)-H(t) \geq x +H(u)-H(t)\\&\geq x+ \inf_{s \leq \bar t \leq \bar u \leq v} \big(H(\bar u)-H(\bar t)\big),
\end{align*}
since the integrand is at least 0 since  $X_{s,r}(x)\leq x \leq 0$ for $r \in [t,u]$. This holds true for all $u$ as above and therefore the assertion holds for $x \leq 0$.

The case of $x>0$ follows by monotonicity from Lemma \ref{monot}.
\end{proof}

Proposition~\ref{Xinf} implies that the function
$t \mapsto X_{s,t}(x)$ cannot reach $-\infty$ in finite time. Further,
if $X_{s,t}(x)\geq A$ for some $t \in [s,v]$ and
$A \geq v+ \inf_{s \leq \kappa \leq u \leq v}\big(
H(u)-H(\kappa)\big)$, then
$$
\inf_{t\leq u \leq v}X_{s,u}(x)\geq A + \inf_{s \leq \kappa \leq u \leq v}\big( H(u)-H(\kappa)\big)
$$
by the same argument as was used in Proposition~\ref{Xinf}, since
$y(y-r)\geq 0$ for $y \geq v$ and $r \in [s,v]$. This result implies that
$\lim_{t \uparrow \beta_s(x)}X_{s,t}(x)=\infty$ whenever
$\beta_s(x)<\infty$. The function $x \mapsto \beta_s(x)$ is
non-decreasing for each $s \geq 0$.

We will write $X_{s,t}(x)=\infty$ whenever $t \geq \beta_s(x)$.

\begin{prp}\label{bfinite}
  For any $s \geq 0$, we have $\beta_s(x)<\infty$ for sufficiently
  large $x$.
\end{prp}
\begin{proof}
  To see this, let $H_*$ be a lower bound of $H(t)-H(s)$ for
  $t \in [s,s+1]$ and let $x$ be so large that the solution of
  $Y'=\frac 12 Y^2$ with initial condition $x+H_*-1$ at time $s$ blows
  up before time $s+1$. In addition, assume that $x\geq 3-H_*$. Let
  $X(t):=X_{s ,t}(x)$, $t \geq s$. Then $X(s)> Y(s)$ since
  $H_*\leq 0$. Let $\tau$ be the infimum of all $t \geq s$ such that
  $X(t)=Y(t)$ and assume that $\tau \leq s+1$. For $u\leq \tau$ we
  have $X(u)\geq Y(u)\geq x+H_*-1$ since $Y$ is increasing.  Then, for
  $u \in [s,\tau]$,
\begin{align*}
X(u) & \big(X(u)-1\big)=\frac 12 X^2(u)+\frac 12
       X(u)\big(X(u)-2\big)\geq \frac 12 Y^2(u)\\
  & +\frac 12 X(u)\big(X(u)-2\big)\geq \frac 12 Y^2(u),
\end{align*}
since $X(u)-2\geq Y(u)-2 \geq 0$. Hence
\begin{align*}
  X(\tau) & =x+\int_s^\tau X(u)\big(X(u)-1\big)\,\d u +H(\tau)-H(s)\\
  & > x+H_* -1 +\frac 12 \int_
s^\tau Y^2(u)\d u =Y(\tau),
\end{align*}
which is a contradiction.
\end{proof}

From Proposition~\ref{bfinite} it follows by monotonicity that the set
\begin{equation} \label{I1s}
  I_1(s) := \{ x\in \R \, | \, \beta_s(x)<\infty\}
\end{equation}  
is an interval which is unbounded above and non-empty.

Now we define 

$$
x_{\inf}(s):=\inf\{x \in \R:\,\beta_s(x)<\infty\}\in [-\infty,\infty).
$$

\begin{prp}\label{Rsoln} 
  \[
    x_{\inf}(t) = X_{s,t} ( x_{\inf} (s) )
  \]
  whenever $s>t$ and $x_{\inf}(s)>-\infty$.  
\end{prp}

\begin{proof}
  Assume that for some $t > s$ we have that
  $ X_{s,t} ( x_{\inf} (s) ) < x_{\inf}(t)$. Consider
  $x \in ( X_{s,t} ( x_{\inf} (s) ), x_{\inf}(t))$. Then by
  monotonicity there exists $y$ such that $x=X_{s,t}(y)$ and
  $y> x_{\inf}(s)$. Hence $X_{s,t}(y)$ blows up in finite time but
  since $x < x_{\inf}(t)$, this is impossible. The argument for the
  case $ X_{s,t} ( x_{\inf} (s) ) > x_{\inf}(t)$ is similar.
\end{proof}  





It is obvious from Proposition~\ref{Rsoln} that
$X_{s,t}(x_{\inf}(s))$ is defined for all $t>s$.    

In particular, the set $I_1(s)$ of (\ref{I1s}) is open; depending on
$H$ it may be all of $\R$ or a proper (non-empty) subset of $\R$.

Now we return to our initial model \eqref{SDE}. A Wiener process $W$
has continuous paths and satisfies $W(0)=0$, so that our general
statements above can be applied for (almost) every
$\omega \in \Omega$.

\begin{thm}\label{1main}
  For every $s\geq 0$ and $x \in \R$ the probability that either
  $\beta_s(x)<\infty$ or $\lim_{t \to \infty}X_{s,t}(x)=0$ equals one,
  i.e., almost surely, either the solution blows up in finite time or
  it converges to 0.
\end{thm}

\begin{rem}\label{theremark}
  Theorem \ref{1main} implies that for each $s \geq 0$ and $x \in \R$
  the set $\Omega_{s,x}\subset \Omega$ consisting of those
  $\omega \in \Omega$ for which the solution starting at $(s,x)$
  neither blows up nor converges to zero has measure 0. For $s\geq 0$
  denote the union of all $\Omega_{s,x},\,x \in \R$ by $\Omega_s$.
  Since this is an uncountable union it does not follow that
  $\Omega_s$ has measure 0. Indeed we will see in the following that
  $\P(\Omega_s)>0$ for every $s \geq 0$. In order to obtain a better
  understanding of the set $\Omega_s$ we will use the notation
  $$
  I_1(s,\omega):=\{x \in \R\, |\beta_s(x)<\infty\}
  $$
  corresponding to a realisation $\omega$, consistent with
  (\ref{I1s}). We saw that $I_1(s,\omega)$ is a non-empty open interval
  which is unbounded from above.  Similarly, we set
  \begin{equation}\label{I2s}
    I_2(s, \omega) := \{ x \in \R \, | \,  X_{s,t}(x) \rightarrow  0
    \hbox{ as } t \rightarrow \infty\}.
  \end{equation}
  Again, by monotonicity of solutions, $I_2(s,\omega)$ is a (possibly
  empty) interval lying below $I_1(s,\omega)$. Using Theorem
  \ref{1main} and a Fubini-type argument we will show that
  $I_2(s,\omega)$ is unbounded below whenever $I_2(s,\omega)$ is
  nonempty, that $I_1(s,\omega)$ is unbounded below whenever
  $I_2(s,\omega)$ is empty and that the complement of the (disjoint)
  union of these intervals can contain at most one point, almost
  surely.

  Let
  $$
\zeta(s,\omega):=\inf I_2(s,\omega) \mbox{ if } I_2(s,\omega)\neq \emptyset
$$
and $\zeta(s,\omega):=-\infty$ otherwise. If
$P\big(\zeta(s,\omega)>-\infty\big)>0$, then, by monotone convergence,
there exists some $x \in \R$ such that
$\P\big(\zeta(s,\omega)>x\big)>0$. On the set
$\{\omega: \zeta(s,\omega)>x\}$ the number $x$ lies below the set
$I_2(s,\omega)$ and, a fortiori, also below $I_1(s,\omega)$, so $x$ is
not in the union of the two intervals which contradicts Theorem
\ref{1main}, so $\zeta(s,\omega)=-\infty$ almost surely, so
$I_2(s,\omega)$ is unbounded below whenever $I_2(s,\omega)$ is
nonempty.

Almost exactly the same argument shows that $\inf I_1(s,\omega)=-\infty$ on the set where   $I_2(s,\omega)$ is empty. The fact that, almost surely,  the complement of the union of $I_1(s,\omega)$ and  $I_2(s,\omega)$
contains at most one point is clear when $I_2(s,\omega)$ is empty (and hence $I_1(s,\omega)=\R$). Otherwise, assuming that $\sup I_2(s,\omega)<\inf I_1(s,\omega)$ with strictly positive probability, by Fubini's theorem, there must
exist some (deterministic) $x \in \R$ for which $\P\big(x \in \big(\sup I_2(s,\omega),\inf I_1(s,\omega)\big)>0$ so, in particular, $x$ is not contained in the union of $I_1(s,\omega)$ and $I_2(s,\omega)$ which again contradicts Theorem \ref{1main}.

  We set
  \[
    \RR(s,\omega)= x_{\inf}(s)
  \]  
  for a realisation $\omega \in \Omega$, i.e. it   is the
  infimum of $I_1(s,\omega)$ (which necessarily coincides with the
  supremum of $I_2(s,\omega)$). Note that $\RR(s,\omega)=-\infty$ iff
  $I_2(s,\omega)$ is empty iff $I_1(s,\omega)=\R$. We know that
  $\RR(s,\omega)$ does not belong to $I_1(s,\omega)$ but at the
  moment, it is unclear if $\RR(s,\omega)$ belongs to
  $I_2(s,\omega)$. We will see later that it almost surely
  doesn't. Therefore (almost surely), $\omega \in \Omega_s$ iff
  $\RR(s,\omega)>-\infty$ and $\omega \in \Omega_{s,x}$ iff
  $\RR(s,\omega)=x$. In particular,
  $\P(\Omega_s)= \P(\RR(s,\omega)>-\infty)$ and we will see in Theorem
  \ref{Theorem4} that this quantity converges to 1 as $s\to \infty$.
\end{rem}


We will denote the blow-up probability of the solution starting at $(s,x)$ by $B(s,x)$, i.e.
$$
B(s,x):= \P\big(\beta_s(x)<\infty\big).
$$
Clearly, the function $x \mapsto B(s,x)$ is non-decreasing and converges to 1 as $x \to \infty$ for every $s \ge 0$ (by the general properties stated above).  
We have the following result concerning the asymptotics of $B(s,x)$ as $s \to \infty$.
\begin{thm}\label{Theorem2}
  For every $z \in \R$ we have
  $$
\lim_{s \to \infty}B\Big(s,s+z\frac \sigma{\sqrt{2s}}\Big)=\Phi(z):=\frac 1{\sqrt{2\pi}}\int_{-\infty}^z\exp\Big\{ - \frac {y^2}2\Big\}\,\d y. 
  $$
\end{thm}

Note that the convergence is even uniform in $z$ since both sides are non-decreasing, take values in $[0,1]$, and $\Phi$ is continuous. In what follows, we will call the subset $\{(s,s): s\geq 0\}$ the {\em diagonal}. Theorem \ref{Theorem2} (together with Theorem \ref{1main}) says that, for large $s$,
solutions starting slightly above the diagonal will blow-up with high probability while solutions starting slightly below the diagonal will converge to 0 with high probability. 

%
%
%
The theorem above tells us, in particular, that $B(s,0)$ converges
to 0 as $s \to \infty$. This does not mean that $B(s,x)$ converges to
0 as $x \to -\infty$ for fixed $s$. In fact it doesn't. Let
$B(s,-\infty)$ be the limit of $B(s,x)$ as $x \to -\infty$ which exists since $x \mapsto B(s,x)$ is non-decreasing.
\begin{thm}\label{Theorem3}
  $$
B(s,-\infty)>0
$$
for every $s \geq 0$.
\end{thm}

Finally, we investigate the random borderline $\RR$ between blow-up and convergence to 0. 

Recall that $\scr{R}(s)\in [-\infty,\infty)$ for all $s \geq 0$ and
that $\scr{R}(s)=-\infty$ implies $\scr{R}(r)=-\infty$ for every
$r<s$. Further, $X_{s,t}(\scr{R}(s))=\scr{R}(t)$ whenever $t\geq s$
and $\scr{R}(s)>-\infty$, so $\scr{R}(t)$ is a solution of \eqref{SDE}
on the interval where it is larger than $-\infty$. Theorem
\ref{Theorem2} implies that, for each $\varepsilon>0$ and
$\alpha \in (0,1/2)$,
$\lim_{s \to \infty}\P\big( s^\alpha|\scr{R}(s)-s|\geq
\varepsilon\big)=0$ but the following result is stronger.
\begin{thm}\label{Theorem4}
  For each $\alpha \in (0,1/2)$,
  $$
\P\big(\lim_{s \to \infty} s^\alpha|\scr{R}(s)-s|=0\big)=1.
  $$
  \end{thm}


  We will prove the four theorems in the next  three sections.

  In the deterministic case, it follows from the fact that the initial
  value of the repelling river at time 0 is positive, together with
  the fact that the derivative $x'$ in (\ref{deter}) is zero on the
  diagonal $x=t$ that the function $\scr{R}$ lies above the diagonal,
  i.e. $\scr{R}(s)>s$ for all $s \geq 0$. This is not true in the
  stochastic case.
  \begin{prp}\label{osc}
    For each $c\geq 0$ there exist increasing sequences of random times $\rho_n$ and $\kappa_n$ such that, almost surely, $\lim_{n \to \infty}\rho_n=\lim_{n \to \infty}\kappa_n=\infty$ and
    $\scr{R}(\rho_n)>\rho_n+\frac c{\sqrt{\rho_n}}$ and
    $\scr{R}(\kappa_n)<\kappa_n - \frac c{\sqrt{\kappa_n}} $ for all $n \in \N$. 
   \end{prp}

   \begin{proof}
     The events
     \begin{align*}
       A&:=\{\omega:\exists s_0\geq 0,\,\forall s \geq s_0:\, \scr{R}(s)\geq s-\frac c{\sqrt{s}}\},\\
       B&:=\{\omega:\exists s_0\geq 0,\,\forall s \geq s_0:\, \scr{R}(s)\leq s+\frac c{\sqrt{s}}\},
     \end{align*}
     are tail events of the driving Brownian motion $W$, i.e.~$A$ and $B$ are contained in the tail-$\sigma$ algebra $\scr{T}:=\bigcap_{s\geq 0}\sigma\big(W(t)-W(s),\,t \geq s\big)$.
     Kolmogorov's 0-1 law (or Blumenthal's 0-1 law \cite[Theorem 2.7.17]{KS} and time inversion of Brownian motion \cite[Lemma 2.9.4(ii)]{KS}) state that $\scr{T}$ is trivial, so $A$ and $B$ have either probability 0 or 1. Further,
     \begin{align*}
       \P(A)&=\lim_{s_0\to \infty}\P\Big(\RR(s)\geq s-\frac c{\sqrt{s}}\,\forall s\geq s_0 \Big)\leq\liminf_{s_0\to \infty} \P\Big(\RR(s_0)\geq s_0-\frac c{\sqrt{s_0}}\Big)\\
       &=\liminf_{s_0\to \infty}\Big(1-B\Big(s_0,s_0-\frac c{\sqrt{s_0}}\Big)\Big)<1,
     \end{align*}
     where the final inequality follows from Theorem \ref{Theorem2}. Analogously, $P(B)<1$, so 
     both events have probability 0 and the statement of the proposition follows.
   \end{proof}

   \begin{rem}\label{inft}
     Proposition \ref{osc} complements the statement of Theorem \ref{Theorem4}: we have, almost surely,
     $$
\limsup_{s \to \infty} s^{1/2}\big(\RR(s)-s\big)=\infty \mbox{ and } \liminf_{s \to \infty} s^{1/2}\big(\RR(s)-s\big)=-\infty. 
     $$
     \end{rem}
   
   Finally, we provide an explicit asymptotic representation of the unstable river $\scr{R}$ by linearizing around the diagonal.

   \begin{cor}\label{Rbar}
Let $Z_1$ be the unique solution of the SDE
     $$
\d z= \big(tz-1\big)\,\d t+\sigma\,\d W(t),\;t \geq 0,
$$
with random initial condition $Z_1(0)$ chosen such that the second moment of $Z_1(t)$ converges to 0 as $t\to \infty$,
i.e.
$$
Z_1(0)=\int_0^\infty \exp\Big\{ -\frac{u^2}2        \Big\}\,\d u -\sigma \int_0^\infty\exp\Big\{-\frac{u^2}2 \Big\}\,\d W(u), 
$$
so
$$
Z_1(t)=\exp\Big\{\frac{t^2}2 \Big\}\Big(\int_t^\infty\exp\Big\{-\frac{u^2}2 \Big\}\,\d u     -\sigma  \int_t^\infty\exp\Big\{-\frac{u^2}2 \Big\}\,\d W(u)\Big).
$$
Define $Z_0(t)\equiv 0$ and, recursively,
$$
Z_{n+1}(t):=Z_n(t)- \exp\Big\{\frac{t^2}2 \Big\}\Big(\int_t^\infty\exp\Big\{-\frac{u^2}2 \Big\}\big(Z_n^2(u)-Z_{n-1}^2(u)\big)\,\d u\Big)      ,\; n\geq 1,
$$
and $\RR_n(t):=t+Z_n(t)$, $n \in \N_0$, $t\geq 0$. Then, for each $\varepsilon>0$, $n \in \N$,
\begin{equation}\label{star}
\big|\RR_n(t)-\RR(t)\big|=o(t^{-\frac 32 n -\frac 12 +\varepsilon}),\;t \to \infty,
\end{equation}
and there exists some $T=T(\omega)\in (0,\infty)$    such that $\RR_n$ converges to $\RR$ uniformly on $[T,\infty)$.
%
\end{cor}

   \begin{proof}
Put $Z(t)=\RR (t)-t$. Then, using Theorem \ref{Theorem4}, $Z$ is the unique solution of the equation
     $$
\d z= \big(tz+z^2-1\big)\,\d t+\sigma\,\d W(t),\;t \geq 0
$$
such that $Z(t)$ converges to 0 as $t\to \infty$ ($Z$ may be $-\infty$ up to some random time $\zeta$). Then, for $t > \zeta$,
\begin{align*}
Z(t)
                 &=\exp\Big\{\frac{t^2}2 \Big\} \Big( -\sigma \int_t^\infty \exp\Big\{-\frac{u^2}2 \Big\}\,\d W(u)-\int_t^\infty  \exp\Big\{-\frac{u^2}2 \Big\}\big(Z^2(u)-1\big)\,\d u     \Big)\\
  &=Z_1(t)-\exp\Big\{\frac{t^2}2 \Big\}  \int_t^\infty  \exp\Big\{-\frac{u^2}2 \Big\}Z^2(u)\,\d u,   
\end{align*}
and therefore
\begin{align*}
  \RR_n(t)-\scr{R}(t)& = \RR_1(t)-\RR(t)+\sum_{k=2}^n\big(\RR_k(t)-\RR_{k-1}(t)\big)\\
&=Z_1(t)-Z(t)-\sum_{k=2}^n \exp\Big\{\frac{t^2}2 \Big\} \int_t^\infty \exp\Big\{-\frac{u^2}2\Big\}\big(Z_{k-1}^2(u)-Z_{k-2}^2(u)\big)\d u\\
                     &=\exp\Big\{\frac{t^2}2 \Big\}  \int_t^\infty  \exp\Big\{-\frac{u^2}2\Big\} \Big( Z^2(u)-Z_{n-1}^2(u)\Big)\d u\\
  &=\exp\Big\{\frac{t^2}2 \Big\}  \int_t^\infty  \exp\Big\{-\frac{u^2}2\Big\} \Big( Z(u)-Z_{n-1}(u)\Big)\Big(Z(u)+Z_{n-1}(u)\Big)\d u.
\end{align*}
Define $\Gamma_n(t):=\sup_{u\geq t}\big| \RR_n(u)-\scr{R}(u)
\big|$. Choose $\varepsilon < 1/2$. By Theorem \ref{Theorem4}, there
exists some $t_0(\varepsilon,\omega)<\infty$ such that
$|Z(u)|\leq u^{-\frac12+\varepsilon}$ for all
$u \geq t_0(\varepsilon,\omega)$. Hence, for
$t \geq t_0(\varepsilon,\omega)$,
\begin{align*}
  \Gamma_n(t)&\leq \exp\Big\{\frac{t^2}2 \Big\}  \Gamma_{n-1}(t)\frac 1t \int_t^\infty u  \exp\Big\{-\frac{u^2}2\Big\}\d u\, \Big( \Gamma_{n-1}(t)+2t^{-\frac 12 +\varepsilon}\Big)\\
             &=\Gamma_{n-1}(t) \frac 1t  \Big(\Gamma_{n-1}(t)+2t^{-\frac 12 +\varepsilon}\Big).
\end{align*}
Let $t \geq t_1(\varepsilon,\omega):= 3\vee t_0(\varepsilon,\omega)$.
Since $\varepsilon<1/2$, by induction,
$\Gamma_{n}(t)\leq t^{-\frac 12 + \varepsilon}$ (this is true for
$n=1$) since
  $$
\Gamma_n(t)\leq 3t^{-\frac 32 +\varepsilon} \Gamma_{n-1}(t)\leq \Gamma_{n-1}(t),\, t \geq t_1(\varepsilon,\omega).
$$
Hence, for $t \geq t_1(\varepsilon,\omega)$, we obtain
$$
\Gamma_n(t)\leq \big(3t^{-\frac
  32+\varepsilon}\Big)^{n-1}\Gamma_1(t)\leq \big(3t^{-\frac
  32+\varepsilon}\big)^{n-1}t^{-2 +2\varepsilon},
$$
so \eqref{star} follows. Further, for
$t \geq T(\omega):=4\vee t_1(\frac 14,\omega)$, we have
$$
\Gamma_n(t)\leq \big(3 \cdot 4^{-\frac 32+\frac 14}\big)^{n-1}
t^{-2+\frac 12}\leq \Big(\frac 34\Big)^{n-1}t^{-\frac 32},
$$
which converges to 0 uniformly on $[T(\omega),\infty)$, so the final
assertion follows.
\end{proof}
%
%
\begin{rem}
Corollary \ref{Rbar} together with Remark \ref{inft} shows that $\RR_1(s)$ is a much better asymptotic approximation to $\RR(s)$ than $s$: we have $\big|\RR_1(s)-\RR(s)\big|=o(s^{-2+\varepsilon})$ while $\big|s-\RR(s)\big|$ is not even of order $s^{-1/2}$.
\end{rem}
  
\begin{rem}
  The recursion in the previous corollary can also be applied when $\sigma=0$. Then, by \eqref{asymp}, $Z(t)=O(\frac 1t)$, hence $\Gamma_1(t)\leq O(t^{-3})$ and $\Gamma_{n}\leq \Gamma_{n-1}(t) O(t^{-2})$
  (where the $O(t^{-2})$ term can be chosen independently of $n$), so $\Gamma_n(t)\leq O\big( \frac 1{t^{2n+1}}\big)$ in accordance with \eqref{asymp}. This approximation is similar to the one proposed by Blais, see \cite{Blais}, but the derivation is different.
\end{rem}

\section{Proofs of Theorem \ref{1main} and Theorem \ref{Theorem2}}

Let us explain the strategy of the proofs. Let
$$
C(s,x):=\P\Big( \lim_{t \to \infty}X_{s,t}(x)=0\Big),\,x\in \R,\,s \geq 0,
$$
so $C(s,x)$ is the probability that $X_{s,t}(x)$ converges to 0 as
$t \to \infty$.  Clearly, $C(s,x)+B(s,x)\leq 1$ since a trajectory
which blows up cannot converge to 0 at the same time. Therefore, the
statement in Theorem \ref{1main} is equivalent to the statement that
$C(s,x)+B(s,x)\geq 1$ for all $x \in \R$ and $s \geq 0$.  We will say
that a statement concerning $s$ is {\em very likely} if its
probability converges to 1 as $s \to \infty$. We will show the
following:
\begin{itemize}
\item The solution with initial condition $(s,x)$ with $x\geq s+1$
  will either blow-up or hit the diagonal for some $t \geq s$. It is
  very likely that blow up happens before hitting the diagonal.  In
  particular, $\lim_{s \to \infty}B(s,s+1)=1$ (Lemma \ref{lemma}). It
  seems plausible that the map $s \mapsto B(s,s+1)$ is non-decreasing
  but we will not prove this and instead only use that
  $\bar B(s):=\inf_{u \geq s}B(u,u+1)$ converges to 1 as
  $s \to \infty$.
\item The solution with initial condition $(s,x)$ with $|x-s|< 1$ will
  hit one of the two lines parallel to the diagonal at distance 1 for
  some $t \geq s$.  We will provide rather precise asymptotic results
  for the probability $p^+(s,x)$ of hitting the upper line before the
  lower line in Lemma \ref{lem2}. Together with Lemma \ref{lemma}, we
  thus obtain a lower bound for the blow-up probability for these
  initial conditions namely $B(s,x) \geq p^+(s,x)\bar B(s)$. These
  bounds will later turn out also to be asymptotic upper bounds but so
  far, we can not exclude the possibility that solutions starting at
  $|x-s|\leq 1$ and hitting the lower set $\{(t,t-1):\,t \geq s\}$
  will still return to a neighborhood of the diagonal many times and
  eventually blow up with large probability (or even almost surely).
\item The solution with initial condition $(s,x)$ with $x \in (0,s-1]$
  reaches the set $\{0,s\}$ in finite time. Let $\rho(s,x)$ be the
  probability that it hits level 0 before level $s$. Lemma
  \ref{schwieriger} says that
  $\lim_{s \to \infty}\inf_{x \leq s-1}\rho(s,x)=1$ which implies that
  $\bar \rho(s):= \inf_{t \geq s}\inf_{x \in (0,s-1]}\rho(t,x)$
  converges to 1 as well.

\item The solution with initial condition $(s,x)$, $x <0$ will hit $0$
  after finite time (irrespective of the starting time $s\geq
  0$). This is Lemma \ref{einfach}.
    \item  Let $\chi(s)$ be the probability that the solution $X_{s,t}(0)$, $t \geq s$ 
      will remain in $[-1,1]$ for all $t \geq s$ and that
      $\lim_{t \to \infty}X_{s,t}(0)=0$. Lemma \ref{gegennull} states that $\lim_{s \to \infty}\chi(s)=1$. Hence $\bar \chi(s):=\inf_{u \geq s}\chi(u)$ converges to 1 as well.
    \end{itemize}
      These statements together imply the claims in Theorems \ref{1main} and \ref{Theorem2} as follows:

      Let $$\tilde p(s,x):=\left\{\begin{array}{ll}p^+(s,x), &|x-s|\leq 1\\
                                    1, &x>s+1\\
                                    0, &x<
                                         s-1.
                                   \end{array} \right.
                                 $$
                                 Then $B(s,x)\geq \tilde p(s,x)\bar B(s)$ for all $x \in \R$, $s \geq 0$.

                                 Further,
               $$C(s,x) \geq \big(1-\tilde p(s,x)\big) \bar \rho(s)\bar \chi(s).$$

        Hence,
        \begin{equation}\label{BC}
          \inf_{x \in \R}\Big(B(s,x)+C(s,x)\Big)\geq \bar B(s)\wedge\big(\bar \rho(s)\bar\chi(s)\big)
          \end{equation}
      which converges to 1 as $s \to \infty$.

      Let $\kappa(s,x)$ be the probability that the solution starting at $(s,x)$ neither blows up nor converges to 0. Equation \eqref{BC} shows that $\lim_{s \to \infty}\sup_{x \in \R}\kappa(s,x)=0$.
      Therefore, for $t\geq s$ and $x \in \R$,
      \begin{equation}\label{kappe}
\kappa(s,x)\leq\int_\R \kappa(t,y)\,\P\big(X_{s,t}(x)\in\d y\big)
\end{equation}
(the ``$\leq$'' is due to the fact that there may be solutions which reach $y$ at time $t$ and which are $-\infty$ at time $s$) converges to 0 as $t \to \infty$, so $\kappa(s,x)=0$ for all $s \geq 0$ and all $x \in \R$ showing that $B(s,x)+C(s,x)\geq 1$ as desired and Theorem \ref{1main} follows.\\

Theorem \ref{Theorem2} is now an easy consequence of these considerations and Lemma \ref{lem2}. Fix $z \in \R$. We saw above that
$$
B\Big(s,s+z\frac \sigma {\sqrt{2s}}\Big)\geq \tilde p\Big(s,s+z\frac \sigma {\sqrt{2s}}\Big)\cdot \bar B(s)
$$
and
$$
B\Big(s,s+z\frac \sigma {\sqrt{2s}}\Big)=1-C\Big(s,s+z\frac \sigma {\sqrt{2s}}\Big)\leq 1-\Big(1-\tilde  p\Big(s,s+z\frac \sigma {\sqrt{2s}}\Big)\Big)\bar \rho(s)\bar \chi(s).
$$
Inserting the limit $\Phi(z)$ of $\tilde  p\Big(s,s+z\frac \sigma {\sqrt{2s}}\Big)$ established in Lemma \ref{lem2} the statement  of Theorem \ref{Theorem2} follows.\\

We will now formulate and prove the lemmas mentioned above. Since the proof of Theorem \ref{Theorem4} will require finer estimates of $B(s,x)$ than those of Theorems \ref{1main} and \ref{Theorem2} we will
formulate the lemmas accordingly. The first
lemma provides a lower bound for the blow-up probability when a trajectory starts slightly above the diagonal. The stated result is in fact more precise than needed to prove the two theorems. It will  be convenient to work with the process
$$
Y_{s,t}(x):=X_{s,t}(x+s)-t,\;0\leq s\leq t.
$$
Note that $Y_{s,.}(x)$ solves the equation
\begin{equation}\label{why}
\d Y(t)=\big(Y(t) (Y(t) +t)-1     \big)\,\d t +\sigma\,\d W(t), \,t\geq s,\,Y(s)=x.
\end{equation}

\begin{lem}\label{lemma}
For each  $\alpha \in [0,1/2)$, we have
$$
B\big(s,s+s^{-\alpha}\big)\geq 1 - \exp\Big\{ -\frac 1{\sigma^2}s^{1-2\alpha}\big(1+o(1)\big)\Big\},\,s \to\infty.
$$
In particular, $\lim_{s \to \infty}\inf_{x \geq s+s^{-\alpha}}B(s,x)=1$.
\end{lem}

\begin{proof} We want to bound the probability $P(\gamma,s):=B(s,s+\gamma)$ of blow-up of $X$ starting from $X_{s,s}(s+\gamma)=s+\gamma$ or, equivalently, of $Y$ starting from $Y_{s,s}(\gamma)=\gamma$ for a given $\gamma>0$ and $s>0$. We compare $Y_{s,t}(\gamma)$ with the solution $\bar Y$ of the equation
\begin{equation}\label{ybar}
\d\,\bar Y(t) =\big(\bar Y(t)(\bar Y(t)+s)-1\big)\,\d t + \sigma\,\d W(t),\,t \geq s,\bar Y(s)=\gamma.
\end{equation}
$\bar Y$ solves an equation with  coefficients independent of $t$ and is therefore easier to analyze than $Y$.
Note that $\bar Y(t) \leq Y(t)$ as long as $Y$ is non-negative. Further, $\bar Y$ blows up to $\infty$ with probability 1 by  Feller's test of explosion, see Lemma \ref{1ddiff}c), applied to $(l,r)=\R$. Therefore,
$$
Q(\gamma,s):=\P\big(\inf_{t \geq s}  \bar Y(t) > 0\big) \leq P(\gamma,s).
$$
Let $b(u):=u(u+s)-1$, and  denote by $p$ the {\em scale function} of $\bar Y$, i.e.
$$
p(x):=\int_0^x \exp\Big\{-\frac 2 {\sigma^2}\int_0^y b(u)\,\d u\Big\}\,\d y=    
\int_0^x \exp\Big\{-\frac 2 {\sigma^2}\Big(\frac{y^3}3+s\frac{y^2}2-y     \Big)\Big\}\,\d y,\,x \in \R.
$$
Then, by Lemma \ref{1ddiff}a) with $r=\infty$ and $l=0$,
$$Q(\gamma,s)=\frac {p(\gamma)}{p(\infty)}.
$$
Note that $p(\infty)<\infty$. We have
$$
p(\infty)=\int_0^\gamma \exp\Big\{-\frac 2 {\sigma^2}\Big(\frac{y^3}3+s\frac{y^2}2-y     \Big)\Big\}\,\d y+\int_\gamma^\infty \exp\Big\{-\frac 2 {\sigma^2}\Big(\frac{y^3}3+s\frac{y^2}2-y     \Big)\Big\}\,\d y=:A+B,
$$
so $Q(\gamma,s)=\frac A{A+B}$. To obtain a lower bound, we estimate $B$ from above and $A$ from below. 
Now we assume  that $s \geq 1$. Then
\begin{align*}
B &\leq \int_\gamma^\infty \exp\Big\{-\frac 2 {\sigma^2}\Big(s\frac{y^2}2-y     \Big)\Big\}\,\d y\\
&=\frac\sigma{\sqrt{2s}}\exp\Big\{  \frac 1{s\,\sigma^2 }  \Big\}\int^\infty_{\frac{\sqrt{2s}\gamma}\sigma -\frac{\sqrt 2}{\sigma \sqrt s}}\exp\Big\{-\frac {z^2}2\Big\}\,\d z,
\end{align*}
and 
\begin{align*}
A &\geq \exp\Big\{-\frac 2{\sigma^2} \frac{\gamma^3}3\Big\}  \int_0^\gamma \exp\Big\{-\frac 2 {\sigma^2}\Big(s\frac{y^2}2-y     \Big)\Big\}\,\d y\\
&    =  \exp\Big\{-\frac 2{\sigma^2} \frac{\gamma^3}3\Big\}  \frac\sigma{\sqrt{2s}}\exp\Big\{  \frac 1{s\,\sigma^2 }  \Big\}\int_{-\frac{\sqrt 2}{\sigma \sqrt s}}^{\frac{\sqrt{2s}\gamma}\sigma -\frac{\sqrt 2}{\sigma \sqrt s}}\exp\Big\{-\frac {z^2}2\Big\}\,\d z.
\end{align*}
Inserting $\gamma=s^{-\alpha}\big(\leq 1\big)$, setting $u:=\frac{\sqrt{2s}\gamma}\sigma -\frac{\sqrt 2}{\sigma \sqrt s}$,  and using the estimate
$$
\int_R^{\infty}\exp\Big\{ -\frac 12 z^2   \Big\}\,\d z \leq \sqrt{\frac \pi 2} \exp\Big\{-\frac 12 R^2\Big\}
$$ 
for any $R \geq 0$, we get
\begin{align*}
  1-Q\big( s^{-\alpha},s\big)&\leq \frac BA \leq \exp\Big\{\frac 23 \sigma^{-2}\Big\} \int_u^{\infty} \exp\Big\{ -\frac 12 z^2   \Big\}\d z\Bigg(\int_0^u \exp\Big\{ -\frac 12 z^2   \Big\}\d z\Bigg)^{-1}\\
  &= \exp\Big\{\frac 23 \sigma^{-2}\Big\} \int_u^{\infty} \exp\Big\{ -\frac 12 z^2   \Big\}\d z\Bigg(\sqrt{\frac \pi 2}-\int_u^\infty \exp\Big\{ -\frac 12 z^2   \Big\}\d z\Bigg)^{-1}\\
  &\leq  \exp\Big\{\frac 23 \sigma^{-2}\Big\} \exp\Big\{-\frac {u^2}2   \Big\}  \Big(1- \exp\Big\{-\frac {u^2}2   \Big\} \Big).\\
                             &\leq  \exp\Big\{\frac 23 \sigma^{-2}\Big\} \exp\Big\{-\sigma^{-2}\big( s^{\frac 12-\alpha}-s^{-\frac 12}\big)^2\Big\}\big(1+o(1)\big),\;s \to \infty.
\end{align*}
Therefore, using the writing $\kappa_1(s)\lesssim \kappa_2(s),\;s \to \infty$ whenever $\limsup_{s \to \infty}\Big(\kappa_1(s)/\kappa_2(s)\Big)\leq 1$,
$$
\log\big(1-P\big(s^{-\alpha},s\big)\big)\leq \log\big(1-Q\big(s^{-\alpha},s\big)\big)\lesssim -\frac 1{\sigma^2}s^{1-2\alpha}, \; s \to \infty,
$$
so the statement of the lemma follows as $\alpha<1/2$.
\end{proof}

\begin{rem}
  In the special case $\alpha=0$, the statement of the previous lemma can be regarded as a (one-sided) {\em large deviations} estimate: the  probability of no blow-up of the process
  starting at $(s,s+1)$ decays (at least) exponentially fast as $s \to \infty$. Note that the exponential rate depends on $\sigma$: the smaller $\sigma$ the faster is the exponential
  decay rate. This is not surprising because in the deterministic limit $\sigma \to 0$, solutions starting from $(s,s+1)$ will certainly blow-up as long as $s$ is not too small.
\end{rem}

\begin{rem}
Together with Theorem \ref{1main} (whose proof is not yet complete), the previous lemma shows, in particular, that $\RR(s)$ is {\em not} contained in the set $I_2(s,\omega)$ introduced in Remark \ref{theremark}, so whenever $\RR(s)>-\infty$ then there is a {\em trichotomy} of asymptotic fates of the solutions starting at time $s$: if $x>\RR(s)$, then $X_{s,t}(x)$ blows up to $\infty$; if $x=\RR(s)$, then $X_{s,t}(x)$ does not blow up but $\limsup_{t \to \infty} X_{s,t}(x)>0$ while for $x<\RR(s)$, $\lim_{t \to \infty}X_{s,t}(x)=0$. We will establish more precise statements on the asymptotics of $\RR(s)$ later (Theorem \ref{Theorem4}).  
  \end{rem}
  
  Next, we study solutions starting within distance 1 of the diagonal. Define $Y$ as in \eqref{why}.

\begin{lem}\label{lem2}
  For $x \in [-1,1]$, $Y(t),\,t \geq s$ with initial condition $Y(s)=x$ will exit the interval $[-1,1]$ in finite time, almost surely.   Let $p^+(s,x)$ be the probability that $Y(t)$, $t \geq s$ exits the interval $[-1,1]$ via 1. Then, for $z \in \R$,
  \begin{equation}\label{limm}
\lim_{s \to \infty}p^+\Big(s,z\frac {\sigma}{\sqrt{2s}}\Big)= \Phi(z):=\frac 1{\sqrt{2\pi}}\int_{-\infty}^z\exp\Big\{-\frac{y^2}2 \Big\}\,\d y.
\end{equation}
Further, for $\alpha \in [0,1/2)$,
$$
p^+\big(s,-s^{-\alpha}\big) \leq \exp\Big\{-\frac 1{\sigma^2}s^{1-2\alpha}\big(1+o(1)\big)    \Big\},\;s \to \infty.
$$
\end{lem}

\begin{proof}
  Instead of $Y$, we first consider the solution $Z_A$ of
 \begin{equation}\label{ZZZ}
\d Z_A(t)=\big(sZ_A(t)+A\big)\,\d t +\sigma \d W(t),\,t \geq s;\, Z_A(s)=x \in [-1,1],
\end{equation}
where $A \in \R$ may depend on $s$ (but neither on $t$ nor on $Z$).
Let
$$
q_A(x):=\int_0^x\exp \Big\{-\frac 2{\sigma^2} \int_0^y\big(us+A\big)\,\d u   \Big\}\,\d y
$$
be the scale function of $Z_A$. Then
\begin{align*}
q_A(x)&=\int_0^x \exp \Big\{-\frac 2{\sigma^2} \Big(\frac{sy^2} 2+Ay\Big)   \Big\}\,\d y\\
&=\int_{\frac{2A}{\sigma \sqrt{2s}}}^{\frac {\sqrt{2s}x}\sigma + \frac{2A}{\sigma\sqrt{2s}}} \exp\Big\{-\frac 12 v^2\Big\}\,\d v \cdot \frac \sigma{\sqrt{2s}}  \exp\Big\{\frac {A^2}{\sigma^2 s}\Big\}.
\end{align*}
By Lemma \ref{1ddiff}a), the probability $p_A(s,x)$ that $Z_A$ exits the interval $[-1,1]$ via 1 equals
\begin{equation}\label{frac}
\frac{q_A(x)-q_A(-1)}{q_A(1)-q_A(-1)}=\frac{\Phi\Big( \frac {\sqrt{2s}x}\sigma +\frac{2A}{\sigma \sqrt{2s}}\Big) - \Phi\Big( -\frac {\sqrt{2s}}\sigma +\frac{2A}{\sigma \sqrt{2s}}\Big)}
{\Phi\Big( \frac {\sqrt{2s}}\sigma +\frac{2A}{\sigma \sqrt{2s}}\Big) - \Phi\Big( -\frac {\sqrt{2s}}\sigma +\frac{2A}{\sigma \sqrt{2s}}\Big)}.
\end{equation}
For given $z \in \R$, define $x=z\frac \sigma{\sqrt{2s}}$. Then, by \eqref{frac},
\begin{equation}\label{assame}
\lim_{s \to \infty} p_A\Big(s, z\frac \sigma{\sqrt{2s}}\Big)=\Phi(z),
\end{equation}
provided that $A=A(s)=o(\sqrt{s})$.

Next, we compare  $Y$ and, for specific functions $A$, $Z_A$  both with the same initial condition $x\in [-1,1]$ at time $s$.
For $A=A(s)=o\big(s^{1/2}\big)$ let
$$
T(s,x,A):=\inf\{t \geq s: |Z_A(t,x)|=1\}
$$ denote the first time after $s$ when  $Z_A$ starting at $Z_A(s)=x\in[-1,1]$ hits $\{-1,1\}$. We will show that, for every such function $A=o\big(s^{1/2}\big)$,
\begin{equation}\label{zuzeigen}
  \zeta_A:=\limsup_{s \to \infty}\Big(\sup_{x \in [-1,1]}\E \big[T(s,x,A)\big]-s\Big)<\infty.
\end{equation}

Once we have established \eqref{zuzeigen}, we apply it to 
$\bar A(s):=s^\alpha$, $\underline A(s):=-s^\alpha-1 $ for some fixed value of $\alpha \in (0,1/2)$.  We estimate the drift function $b(t,y)$ of $Y$ from above
by $b(t,y)=y^2-1+yt= y^2-1+ys+y(t-s) \leq ys + \bar A(s)$ 
and, from below, by $b(t,y)=y^2-1+yt= y^2-1+ys+y(t-s) \geq ys + \underline A(s)$ whenever $t \in [s,s+s^\alpha]$, and $y \in [-1,1]$.

Fix $s >0$ and $z \in \R$ such that $x:=z\frac{\sigma}{\sqrt{2s}}\in[-1,1]$ and consider  $Z_{\bar A}$ and $Z_{\underline A}$ with initial condition $x$ at time $s$.
By comparison, we have $Z_{\underline A}(t)\leq Y(t) \leq Z_{\overline A}(t)$ for $t \in [s,s+s^\alpha]$ up to the minimum of the exit times $T(s,x,\bar A)$ and  $T(s,x,\underline A)$ of
$Z_{\overline A}$ and $Z_{\underline A}$ from $[-1,1]$, where $Y$ denotes the solution of \eqref{why} with the same initial condition $x$.  We saw that, as $s \to \infty$, the probability of exiting the
interval $[-1,1]$ via 1 is asymptotically the same for   $Z_{\underline A}$ and for $Z_{\bar A}$. In order to ensure that this is also true for $Y$ we have to show that the exit times from $[-1,1]$
of these processes are at most $s+s^{\alpha}$ with probabilities converging to 1 as $s \to \infty$. We have
\begin{align*}
  p^+(s,x)&\geq \P\big(Z_{\underline{A}}\mbox{ exits }[-1,1] \mbox{ via } 1 \mbox{ before or at time } s+s^{\alpha}\big)\\
  &=\P\big( Z_{\underline{A}}\mbox{ exits }[-1,1] \mbox{ via } 1\big)-  \P\big(Z_{\underline{A}}\mbox{ exits }[-1,1] \mbox{ via } 1 \mbox{ after time }s+s^{\alpha}\big).
\end{align*}
Define $p^-(s,x)$ as the probability that $Y$ starting at $x\in [-1,1]$ at time $s$ exits $[-1,1]$ via -1 (clearly, $p^+(s,x)+p^-(s,x)\leq 1$ but it is not yet clear that there is equality). Then, analogously,
\begin{align*}
  p^-(s,x)&\geq \P\big(Z_{\bar{A}}\mbox{ exits }[-1,1] \mbox{ via } -1 \mbox{ before or at time } s+s^{\alpha}\big)\\
  &=\P\big( Z_{\bar{A}}\mbox{ exits }[-1,1] \mbox{ via } -1\big)-  \P\big(Z_{\bar{A}}\mbox{ exits }[-1,1] \mbox{ via } -1 \mbox{ after time }s+s^{\alpha}\big).
\end{align*}  
By Markov's inequality, \eqref{assame} and \eqref{zuzeigen}, we have
$$
 p^+(s,x)+ p^-(s,x)\geq 1+o(1)-\frac{\zeta_{\bar A}}{s^{\alpha}}- \frac{\zeta_{\underline A}}{s^{\alpha}}\to 1,\;s \to \infty,
 $$
 and the same argument as in \eqref{kappe} shows that we actually have $ p^+(s,x)+ p^-(s,x)=1$ for every $|x|\leq 1$ and $s$, so the very first statement in the lemma follows. Further,
 \begin{align*}
   1-\P\big( Z_{\underline{A}}\mbox{ exits }[-1,1] \mbox{ via } -1\big)+o(1)&\geq 1-p^-\big(s,z\frac {\sigma}{\sqrt{2s}}\big)\\
   &=p^+\big(s,z\frac {\sigma}{\sqrt{2s}}\big)\geq \P\big( Z_{\bar{A}}\mbox{ exits }[-1,1] \mbox{ via } 1\big)+o(1),
\end{align*}
and therefore, by \eqref{assame}, \eqref{limm} follows.


It remains to prove \eqref{zuzeigen}. Define
$$
V(t):=\Big(Z_A(t,x)+\frac As\Big)^2,\,t \geq s.
$$
and assume that $s$ is so large that $\frac As \in (-1,1)$. Then, by It\^o's formula, for $t \geq s$,
$$
V(t)=\Big(x+\frac As\Big)^2+2s\int_s^t V(u)\,\d u+M_t+\sigma^2 (t-s),
$$
where $M_t$, $t \geq s$ is a continuous local martingale satisfying $M_s=0$. Let $\tau$ be the first time after $s$ when $V$ attains the value 4. Then, since $M_{t\wedge \tau}$ is a martingale starting at 0 and since $V \geq 0$, we obtain
$$
4 \geq \Big(x+\frac As\Big)^2 + \sigma^2 \Big(\E [\tau]-s\Big),
$$
so
$$
\E [\tau] \leq s+ \frac 4{\sigma^2}.  
$$
This implies $\E \big[T(s,x,A)\big]-s \leq \E [\tau]-s \leq \frac 4{\sigma^2}$ for all $x \in [-1,1]$. 
This completes the proof of the first part of the lemma.

The final claim in the lemma   follows easily from \eqref{frac} :
\begin{align*}
p^+\big(s,-s^{-\alpha}\big)&\leq p_0(s,-s^{-\alpha}\big)=\frac{\Phi\Big(-\frac1\sigma \sqrt{2s}s^{-\alpha}\Big)-\Phi\Big(-\frac1\sigma \sqrt{2s}\Big)}{\Phi\Big(\frac1\sigma \sqrt{2s}\Big)-\Phi\Big(-\frac1\sigma \sqrt{2s}\Big)}\\
&\sim  \Phi\Big(-\frac1\sigma \sqrt{2}s^{\frac 12-\alpha}\Big)  \leq\exp\Big\{-\frac 1{\sigma^2}s^{1-2\alpha}\big(1+o(1)\big)    \Big\}.
\end{align*}

\end{proof}

\begin{rem}
  It should not come as a surprise that the upper bound for the expected exit time computed in the previous proof blows up as $\sigma \to 0$. After all, in the deterministic case, for $s$ sufficiently large, there exists an initial condition $(s,x)$ with $x \in [s-1,s+1]$ for which the solution stays in the interval $(s-1,s+1)$ forever.
\end{rem}

The two previous lemmas together provide an asymptotic lower bound for the blow-up probability for initial conditions of distance at most 1 from the diagonal:  
Lemma \ref{lem2} provides an asymptotic bound for the probability of reaching $s+1$ before reaching $s-1$ and Lemma \ref{lemma} states that once we hit $s+1$ the probability of blow-up converges to 1 as $s \to \infty$.

Our next aim is to show that all solutions starting below 0 will almost surely hit 0 later (Lemma \ref{einfach}) and that with probability converging to 1 as $s \to \infty$, trajectories starting from $(s,s-s^{-\alpha})$ will hit 0 before hitting level $s$ whenever $\alpha \in [0,1/2)$ (Lemma \ref{schwieriger}).

\begin{lem}\label{einfach}
For every $x \in \R$ and $s \geq 0$, we have, almost surely, $\limsup_{t \to \infty}X_{s,t}(x)\geq 0$.
\end{lem}

\begin{proof}
  Fix $x \in \R$ and $s \geq 0$. If $t \geq s$  and $X_{s,t}(x)<0$, then, for  $\tau:=\inf \{u \geq t:\, X_{s,u}(x)\geq 0\}$ and $u \in [t,\tau)$, we have
  \begin{align*}
    0 \geq X_{s,u}(x)&=X_{s,t}(x)+\int_s^uX_{s,v}(x)\big( X_{s,v}(x)-v\big)\,\d v +\sigma\big(W(u)-W(s)\big)\\
    &\geq X_{s,t}(x)   +\sigma\big(W(u)-W(s)\big).
\end{align*}
Since $\limsup_{u \to \infty} \big(W(u)-W(s)\big) =\infty$, we obtain $\tau<\infty$ almost surely, so the claim follows.
\end{proof}

\begin{lem}\label{schwieriger}
  For $s \geq 0$ and $x \in [0,s]$ let $\rho(s,x)$ be the probability that $X_{s,t}(x)$ hits level 0 before level $s$. Then, for $\alpha \in [0,1/2)$,
  \begin{equation}\label{claim1}
1-\rho(s,s-s^{-\alpha})\leq \exp\Big\{-\frac 1{\sigma^2}s^{1-2\alpha}\big(1+o(1)\big)\Big\} 
  \end{equation}
  and
  \begin{equation}\label{claim2}
1-\rho(s,1)\leq \exp\Big\{-\frac 1{3\sigma^2}s^3\big(1+o(1)\big)\Big\}. 
  \end{equation}
\end{lem}

\begin{proof}
  The proof is similar to that of Lemma \ref{lemma} but this time it is more convenient to work with the process $X$ directly rather than with the transformed process $Y$.

  Fix $s > 0$ and $\gamma \in (0,s)$. Let $\bar X(t)$, $t \geq s$ solve
  $$
\d \bar X(t)=\bar X(t)\big(\bar X(t)-s\big)\,\d t+\sigma \d W(t),\;t \geq s,\,\bar X(s)=s-\gamma.
$$
Then, almost surely, $\bar X(t) \geq X(t)$ as long as $\bar X$ is non-negative. Denote by $p$ the scale function of $\bar X$, i.e.
$$
p(x):=\int_0^x\exp \Big\{ -\frac 2 {\sigma^2}\int_0^y b(u)\,\d u\Big\}\,\d y, \; x \in \R,
$$
where
$$
b(u)= u(u-s),
$$
so
$$
p(x)=\int_0^x\exp \Big\{ -\frac 2 {\sigma^2}\Big(\frac 13 y^3 - \frac s2 y^2\Big)  \Big\}\,\d y=s\int_0^{x/s}\exp\Big\{\frac 2 {\sigma^2} s^3\Big(\frac{z^2}2-\frac{z^3}3  \Big)   \Big\}\,\d z  , \; x \in \R.
$$
By Lemma \ref{1ddiff}a), the probability that $\bar X$ leaves the interval $[0,s]$ via $s$ equals
\begin{equation}\label{vgl}
\frac {p(s-\gamma)}{p(s)}\geq 1-\rho(s,s-\gamma).
\end{equation}
Since $p'$ is non-decreasing on $[0,s]$, we have
\begin{align*}
  p(s-\gamma)&\leq s \exp\Big\{\frac 2 {\sigma^2} s^3\Big(1-\frac\gamma s\Big)^2\Big(\frac 12 -\frac{1-\frac \gamma s}3 \Big)   \Big\}\\
  &= s\exp\Big\{\frac 2 {\sigma^2} s^3\Big(\frac 16 -\frac 12 \frac {\gamma^2}{s^2}+\frac 13 \frac {\gamma^3}{s^3}\Big)\Big\}
\end{align*}
and, since $\frac{z^2}2-\frac{z^3}3\geq \frac 16-\frac 12(1-z)^2$ for $z \in  [0,1]$, we obtain
\begin{align*}
  p(s)-p(s-\gamma)&\geq s\int_{1-\frac \gamma s}^1 \exp\Big\{ \frac 2 {\sigma^2} s^3\Big(\frac 16-\frac 12 (1-z)^2\Big)\Big\}\,\d z\\
&= \frac {\sigma}{\sqrt{2}} \exp\Big\{ \frac 2 {\sigma^2} \frac {s^3}6\Big\} s^{-\frac 12}\int_0^{\frac {\sqrt{2}}\sigma s^{1/2}\gamma} \exp\Big\{-\frac 1 2 u^2\Big\}\,\d u.
\end{align*}
Therefore, the probability that $\bar X$ leaves the interval $[0,s]$ via $s$ equals
\begin{align*}
  \frac{p(s-\gamma)}{p(s)}&=\frac{p(s-\gamma)}{p(s-\gamma)+p(s)-p(s-\gamma)}\leq \frac{p(s-\gamma)}{p(s)-p(s-\gamma)}\\
                          &\leq s^{3/2}\frac{\exp\Big\{-\frac {\gamma^2s}{\sigma^2}+\frac 23 \frac {\gamma^3}{\sigma^2}\Big\}}{\frac{\sigma}{\sqrt{2}} \int_0^{\frac {\sqrt{2}}\sigma s^{1/2}\gamma} \exp\Big\{-\frac 1 2 u^2\Big\}\,\d u}.
\end{align*}
Inserting $\gamma=s^{-\alpha}$, observing that the denominator converges to a finite constant as $s \to \infty$, and using \eqref{vgl},  we obtain   \eqref{claim1}.
Inserting $\gamma=s-1$, we obtain \eqref{claim2}.
\end{proof}

\begin{lem}\label{gegennull}
  For $s \geq 0$,   let
  $$
  \chi(s):=\P\big(\big\{\lim_{t \to \infty} X_{s,t}(0)=0\big\}\cap \big\{\sup_{t \geq s}\big|X_{s,t}(0)\big|<1\big\}\big).
  $$
  Then
  $$
  \lim_{s \to \infty}\chi(s)=1.
  $$
\end{lem}

\begin{proof} Fix $s \geq 0$ and let $A\geq 0$. Define the processes $X$  and $X_A$ on $[s,\infty)$ as follows:
  \begin{align*}
\d X(t)&=X(t)\big(X(t)-t\big)\,\d t +\sigma \d W(t), \quad X(s)=0,\\
\d X_A(t)&=\big(A-tX_A(t)\big)\,\d t +\sigma \d W(t),\quad X_A(s)=0.
\end{align*}

Let $\bar \tau:= \inf\big\{t \geq s:\, X(t)\geq 1\big\}$, $\underline \tau:= \inf\big\{t \geq s:\,X(t)\leq -1\big\}$,      $\bar \tau_A := \inf\big\{t \geq s:\, X_A(t)\geq 1\big\}$ and $\underline \tau_A:= \inf\big\{t \geq s:\,X_A(t)\leq -1\big\}$. Note that
$X(t)\leq  X_1(t)$ for $t \in [s, \bar \tau_1\wedge \underline \tau_1]$ and  $X(t)\geq  X_0(t)$ and $X_1(t)\geq X_0(t)$ for $t \geq s$ (here, $X_0(t)$ and $X_1(t)$ correspond to $X_A(t)$ with $A=0$ and $A=1$, respectively).   
Therefore, it suffices to show that
\begin{itemize}
\item [i)] $\lim_{s \to \infty}\P\big( \bar \tau_1 =\infty\big)=1$.
\item [ii)] $\limsup_{t \to \infty} X_1(t)\leq 0$ almost surely for each $s \geq 2$.
\item [iii)] $\lim_{s \to \infty}\P\big( \underline \tau_0 =\infty\big)=1$.
\item [iv)] $\liminf_{t \to \infty} X_0(t)\geq 0$ almost surely for each $s \geq 2$.
\end{itemize}

The process $X_A$ solves an affine SDE. Its solution is given by
\begin{equation}\label{XA}
X_A(t)=A\int_s^t \exp\Big\{ \frac {u^2}2-\frac {t^2}2\Big\}\,\d u +\sigma \exp\Big\{-\frac {t^2}2    \Big\} \int_s^t \exp\Big\{ \frac {u^2}2\Big\}\,\d W(u),\;t \geq s.
\end{equation}
We will now assume that $s \geq 2$. Observe that, for $\kappa >0$,
\begin{align}\label{kette}
\int_s^t  \exp\big\{\kappa u^2\big\}\d u &\leq \int_{t-1}^t  \exp\big\{\kappa u^2\big\}\d u + \int_s^{(t-1)\vee s}  \exp\big\{\kappa u^2\big\}\d u \nonumber\\
  &\leq  \int_{t-1}^t \frac u{t-1} \exp\big\{\kappa u^2\big\}\d u + \int_s^{(t-1)\vee s} \frac us \exp\big\{\kappa u^2\big\}\d u\\
  &\leq  \frac 1{2\kappa (t-1)}\exp\big\{ \kappa t^2\big\}+\frac 1{2\kappa s}\exp\big\{\kappa (t-1)^2\big\}.\nonumber
\end{align}
Therefore, the first term in \eqref{XA} is bounded from below by 0 and from above by
\begin{equation}\label{first}
A \Big(\frac 1{t-1}+\frac 1s\exp\big\{-t+\frac 12\big\}\Big)
\end{equation}
for every $t \geq s\geq 2$. The stochastic integral in \eqref{XA} is a continuous martingale and hence a time-changed Brownian motion (\cite[Theorem 3.4.6]{KS}): there exists a Brownian motion $\overline{W}$ such that, for $t \geq s$, 
$$
\int_s^t \exp\Big\{ \frac {u^2}2\Big\}\,\d W(u)=\overline{W}\Big(\int_s^t\exp\big\{u^2\big\}\,\d u\Big).
$$
The law of the iterated logarithm for Brownian motion (\cite[Theorem 2.9.23]{KS}) implies that there exists a random variable $\overline{C}(\omega)$ such that, for every $r \geq 0$,
$$
\overline{W}(r) \leq \overline{C}(\omega)+ 2\sqrt{r}\log^+\log^+r,
$$
where $\log^+r:=0\vee \log r$. Hence, for $t \geq s \geq 2$,
$$
X_A(t)\leq A \Big(\frac 1{t-1}+\frac 1s\exp\big\{-t+\frac 12\big\}\Big)            +\sigma \exp\Big\{-\frac {t^2}2    \Big\} \Bigg( \overline{C}(\omega)+ 2\sqrt{\int_s^t  \exp\big\{u^2\big\}\d u}\, 2 \log t\Bigg).
$$
Noting that $\sqrt{a+b}\leq \sqrt{a}+\sqrt{b}$ whenever $a,b\geq 0$, we obtain, using \eqref{kette} with $\kappa =1$, 
\begin{align*}
X_A(t)
  \leq  &A \Big(\frac 1{t-1}+\frac 1s\exp\Big\{-t+\frac 12\Big\}\Big)\\
    &+\sigma \exp\Big\{-\frac {t^2}2    \Big\}\,\overline{C}(\omega)+\frac 4 {\sqrt{2}}\sigma \log t \Big(  \frac 1{\sqrt{t-1}}+ \frac 1{\sqrt{s}}\exp\Big\{-t+\frac 12 \Big\}  \ \Big).
\end{align*}
Clearly, $\limsup_{t \to \infty}X_A(t)\leq 0$ almost surely and $\sup_{t \geq s}X_A(t)$ converges to 0 as $s \to \infty$ in probability, so i) and ii) follow taking $A=1$. Observing that
the processes $X_0$ and $-X_0$ have the same law, iii) and iv)  
follow as well and the proof is complete.
\end{proof}

The following corollary is not needed to prove Theorems \ref{1main} and \ref{Theorem2} but will be needed to prove Theorem \ref{Theorem4}.

\begin{cor}\label{coneu}
   $$
B(s,0)\leq   \exp\Big\{-\frac 1{3\sigma^2}s^3\big(1+o(1)\big)\Big\}.
$$
Further, for $\alpha \in [0,1/2)$,
  $$
  B\big(s,s-s^{-\alpha}\big)\leq \exp\Big\{-\frac 1{\sigma^2}s^{1-2\alpha}\big(1+o(1)\big)\Big\}.
  $$
\end{cor}

\begin{proof}
  Define $p^+$ as in the statement of Lemma \ref{lem2}. Then
  $$
  B\big(s,s-s^{-\alpha}\big)\leq p^+\big(s,-s^{-\alpha}\big)+\sup_{u \geq s}B(u,u-1).
  $$
  The last statement in Lemma  \ref{lem2} says that $p^+\big(s,-s^{-\alpha}\big) \leq  \exp\Big\{-\frac 1{\sigma^2}s^{1-2\alpha}\big(1+o(1)\big)\Big\}$. Further, by \eqref{claim1} with $\alpha=0$,
  $$
B(u,u-1)\leq \exp\Big\{-\frac 1{\sigma^2}u   \big(1+o(1)\big) \Big\}+\sup_{r \geq u}B(r,0),
$$
so the second statement in the corollary follows once we have proved the first one.\\

Let $\hat B(s):=\sup_{r \geq s}B(r,0)$ and
define $\chi(s)$ as in Lemma \ref{gegennull}. Then, by Lemma  \ref{gegennull}, $\bar \chi(s):=\inf_{u \geq s}\chi(u)$ converges to 1 as $s \to \infty$. For fixed $s\geq 0$, define the (possibly infinite)
stopping times 
$\lambda_0=s$, $\tau_i=\inf\big\{t \geq \lambda_{i-1}:\,X_{s,t}(0)=1\big\}$, $\lambda_i=\inf\big\{t \geq \tau_i:\,X_{s,t}(0)=0\big\},$ $i \geq 1$. If $\tau_1=\infty$, then there is no blow-up. Note that
$\P\big(\tau_1<\infty\big)\leq 1-\bar \chi(s)$. 
If $\tau_1<\infty$, then the solution after time $\tau_1$ hits $s$ before returning to 0 with probability at most $1-\inf_{u \geq s}\rho(u,1)$. Therefore,   
$$
B(s,0)\leq\P\big(\tau_1<\infty\big) \Big(  \big(1-\inf_{u \geq s}\rho(u,1)\big) +\hat B(s) \Big) \leq   \big( 1-\bar \chi(s)\big)  \Big(  \big(1-\inf_{u \geq s}\rho(u,1)\big) +\hat B(s) \Big)
$$
Inserting the bound \eqref{claim2}, we obtain
$$
B(s,0) \leq \hat B(s) \leq \frac 1{\bar \chi(s)}\exp \Big\{-\frac 1{3\sigma^2}s^3\big(1+o(1)\big)\Big\}
$$
as desired.
\end{proof}

  \section{Proof of Theorem \ref{Theorem3}}
  \begin{proof}[Proof of Theorem \ref{Theorem3}] Note that if, for some $s \geq 0$, $\beta_s(x)<\infty$ for all $x \in \R$, then  $\beta_t(x)<\infty$ for all $x \in \R$ and all $t <s$. Therefore, the function $s \mapsto B(s,-\infty)$ is non-increasing and it suffices to show that $B(s,-\infty)>0$ for all sufficiently large $s$.

 We start by showing that the expected time it takes for a solution starting from $x<0$ at time $s$ to hit 0 is bounded uniformly in $x$ and $s$.   Fix $s\geq 0$ and $x <0$, let $X$ solve \eqref{SDE} with initial condition $(s,x)$ and let $\bar X$ solve
    $$
\d \bar X(t)= \bar X^2(t)\,\d t+\sigma \d W(t),\,t \geq s, \;\bar X(s)=x.
$$

Let $\tau_X$ and $\tau_{\bar X}$ be the first time after $s$ when $X$ respectively $\bar X$ reaches 0. Then $\tau_X\leq \tau_{\bar X}$ and therefore $\E_x\big[\tau_X\big]\leq \E_x\big[\tau_{\bar X}\big]$.

Let $p$ be the scale function of $\bar X$, i.e.
$$
p(x)=\int_0^x\exp\Big\{-\frac 2{\sigma^2}\int_0^y u^2\,\d u\Big\}\,\d y=\int_0^x\exp\Big\{-\frac 2{\sigma^2} \frac{y^3}3\Big\}\,\d y,\; x \in \R.
$$
Then, for $a<x<0$, the expected time for $\bar X$ to hit $\{a,0\}$ starting from $x$ is
$$
M_{a,0}(x)=\frac{p(0)-p(x)}{p(0)-p(a)} \int_a^x\big(p(y)-p(a)\big)\,m(\d y)+\frac{p(x)-p(a)}{p(0)-p(a)} \int_x^0\big(p(0)-p(y)\big)\,m(\d y),
$$
where the {\em speed measure} $m$ on $(-\infty,0]$ is given by
$$
m(\d y)=\frac 2{\sigma^2p'(x)}\,\d y= \frac 2{\sigma^2}\exp\Big\{ \frac 2{\sigma^2} \frac {y^3}3\Big\}\d y 
$$
(Lemma \ref{1ddiff}b)). Then,
$$
\E_x\big[\tau_{\bar X}\big]=\lim_{a \to -\infty}M_{a,0}(x)
$$
and
$$
\sup_{x \in (-\infty,0)}\E_x\big[\tau_{\bar X}\big]=\lim_{x \to -\infty}\E_x\big[\tau_{\bar X}\big]=\int_{-\infty}^0\big(p(0)-p(y)\big)\,\d m(y)=:D<\infty.
$$
In particular, the probability that $X$ starting at $(s,x)$ with $x<0$ reaches the interval $[0,\infty)$ before or at time $s+2D$ is at least $1/2$ (by Markov's inequality).

Next, we show that $u \mapsto B(u,0)$ is bounded away from 0 on compact intervals. Together with the previous result, this implies the statement in Theorem \ref{Theorem3}.

Fix $s \geq 0$. Then, for $t \in [s,s+1]$,
$$
X_{s,t}(0)\geq b(t-s) + \sigma \big(W(t)-W(s)\big),
$$
where $b:=b(s):=\inf_{x \in \R,\, u\in [s,s+1]}\{x(x-u)\}=-\frac 14(s+1)^2$ is a lower bound for the drift on $[s,s+1]$. Hence,
$$
\P\big( X_{s,s+1}(0)\geq s+1\big) \geq \P\big(\sigma W(1)\geq s+1-b\big)
$$
is bounded away from 0 uniformly for $s$ in a compact interval and, using Theorem \ref{Theorem2}, we obtain
\begin{align*}
B(s,0)&\geq \P\big( X_{s,s+1}(0)\geq s+1\big) B(s+1,s+1)\\
&\geq  \P\big(\sigma W(1)\geq s+1-b\big)  B(s+1,s+1) \sim \frac 12 \P\big(\sigma W(1)\geq s+1-b\big),\,s \to \infty.
\end{align*}
This completes the proof of Theorem \ref{Theorem3}.
    \end{proof}
    \begin{rem}
      The proof can be upgraded to a quantitative asymptotic lower bound for both $B(s,0)$ and $B(s,-\infty)$: inserting the explicit value of $b=b(s)$ we get
      $$
B(s,0)\geq \exp\Big\{-\frac 1{8 \sigma^2}s^4\big(1+o(1)\big)\Big\}
$$
(which does not quite match the upper bound in Corollary \ref{coneu}). The same asymptotic lower bound holds for $B(s,-\infty)$. 
      \end{rem}
    \section{Proof of Theorem \ref{Theorem4}}
 \begin{proof}[Proof of Theorem \ref{Theorem4}]
   For every $s > 0$, we have $\P\big(\RR(s)\in \big[s-s^{-\alpha}, s+s^{-\alpha})\big)= B\big(s,s +s^{-\alpha}\big)-B\big(s,s -s^{-\alpha}\big)$. Therefore, Lemma \ref{lemma} and Corollary \ref{coneu} imply
   $$
   \P\big(\RR(s)\notin \big[s-s^{-\alpha}, s+s^{-\alpha})\big)\leq \exp\Big\{-\frac{s^{1-2\alpha}(1+o(1))}{\sigma^2}\Big\}.
   $$
   The last expression summed over all $s=n\in \N$ is finite and therefore, by the first Borel--Cantelli lemma,  we have $\RR(n)\in \big[n-n^{-\alpha}, n+n^{-\alpha}\big)$ for all but finitely many $n \in \N$, almost surely. It remains to show that this property does not only hold for integers but for all (real) $s \geq s_0$ for some random $s_0$.
 
 Fix $\alpha \in (0,1/2)$ (for which we want to show the result) and let $\alpha_0\in (\alpha,1/2)$. Applying the statement above for $\alpha_0$, we see that 
\begin{equation}\label{zero}
n-n^{-\alpha_0} < \RR(n)<n+n^{-\alpha_0}
\end{equation}
for almost every $\omega \in \Omega$ and every $n \geq n_1:=n_0(\alpha_0,\omega)$. 

Let us show the upper bound for $\RR(s)$.  
Fix a deterministic positive integer $n$.  Let $X_s$, $s\in [0,n]$ be the $[-\infty,\infty)$-valued solution of \eqref{SDE} on $[0,n]$ with final condition $X_n=n+n^{-\alpha_0}$ (note that $X_s$ may take the value $-\infty$ in case the process blows up in finite time in the backward time direction). Define
$$
Z_n(t)=X_{n-t}-n+t,\;t\in [0,1].
$$
Then, $Z_n(t)$  solves 
$$
Z_n(t)=n^{-\alpha_0}+t-\int_0^tZ_n(v)\big(Z_n(v)+n-v\big)\,\d v+\sigma\,W_n(t),\;t \in [0,1]
$$ where $W_n(t)=W(n-t)-W(n)$ is again a Wiener process. Once we know that
\begin{equation}\label{once}
\sum_{n\in\N}\P\big( \sup_{t \in [0,1]}Z_n(t) \geq n^{-\alpha}\big)<\infty
\end{equation}
then the upper bound follows from \eqref{zero} and the first Borel--Cantelli lemma.

To show \eqref{once}, let $Y_n$ solve 
$$
\d Y_n(t)=\big(2-nY_n(t)\big)\,\d t +\sigma\,\d W_n(t),\;Y_n(0)=n^{-\alpha_0}.
$$
 Then, 
\begin{equation}\label{ugl}
Y_n(t) \geq Z_n(t)
\end{equation}
for all $t \in [0,1]$ and
\begin{align}\label{oben}
Y_n(t)&=n^{-\alpha_0}\e^{-nt}+ 2\int_0^t\e^{-n(t-s)}\,\d s+\sigma \int_0^t\e^{-n(t-s)}\,\d W_n(s) \nonumber \\
&\leq n^{-\alpha_0}+\frac 2n +\sigma \int_0^t\e^{-n(t-s)}\,\d W_n(s) \nonumber \\
&\leq 3n^{-\alpha_0} +\sigma \int_0^t\e^{-n(t-s)}\,\d W_n(s).
\end{align}

Before finishing the proof of the upper bound, let us see how things change for the lower bound. For $n \in \N$, define $\tilde Z_n$ like $Z_n$ above but with initial condition $\tilde Z_n(0)=-n^{-\alpha_0}$ instead of $n^{-\alpha_0}$ and let $\tilde Y_n(t)$ solve
$$
\d \tilde Y_n(t)=-n\tilde Y_n(t) \,\d t+\sigma \,\d W_n(t),\;\tilde Y_n(0)=-n^{-\alpha_0}.
$$
Then $\tilde Y_n(t)\leq \tilde Z_n(t)$ as long as $\tilde Y_n\geq -1/2$ and
\begin{equation}\label{unten}
\tilde Y_n(t)=-n^{-\alpha_0}\e^{-nt}+\sigma \int_0^t\e^{-n(t-s)}\,\d W_n(s).
\end{equation}

The  integral $ \int_0^t\e^{-n(t-s)}\,\d W_n(s)$ in \eqref{oben} and \eqref{unten} is the solution $R_n(t)$ of the equation
$$
\d R_n(t)=-n\,\d t +\d W_n(t),\; R_n(0)=0.
$$
For every $p>2$,  Lemma \ref{ESVlemma} implies that
\begin{equation}\label{Rn}
\E \Big[\sup_{0\leq t \leq 1}|R_n(t)|^p\Big]\leq C_p n^{1-\frac p2}. 
\end{equation}
for some constant $C_p$.


Therefore, by \eqref{ugl} and Markov's inequality,
\begin{align*}
\P\big(\sup_{t \in [0,1]} Z_n(t)\geq n^{-\alpha}\big) &\leq \P\big(\sup_{t \in [0,1]} Y_n(t)\geq n^{-\alpha}\big)\\
                                                   &\leq   \P\Big(\sup_{t \in [0,1]} \big|R_n(t)\big|\geq \frac 1\sigma \big(n^{-\alpha}-3n^{-\alpha_0}\big)\Big)\\
                                                   &\leq \Big(\frac 1\sigma \big(n^{-\alpha}-3n^{-\alpha_0}\big)\Big)^{-p}\E\Big[\sup_{0\leq t \leq 1}|R_n(t)|^p\Big]\\
&\leq   C_p  n^{1-\frac p2} \Big(\frac 1\sigma \big(n^{-\alpha}-3n^{-\alpha_0}\big)\Big)^{-p},
\end{align*}
for  $n$ large enough (such that $n^{-\alpha}>3n^{-\alpha_0}$). The last term summed over $n$ is finite provided $p$ is large enough, so 
assertion \eqref{once} holds and the proof of the upper bound is complete. Analogously, for $n$ and $p$ sufficiently large,
$$
\P\big(\inf_{t \in [0,1]} \tilde Z_n(t)\leq -n^{-\alpha}\big) \leq  C_p  n^{1-\frac p2} \Big(\frac 1\sigma \big(n^{-\alpha}-n^{-\alpha_0}\big)\Big)^{-p}
$$
is summable over $n$. Therefore,
$$
\lim_{s_0 \to \infty}\P\Big( \RR(s) \in \big(s-s^{-\alpha},s+ s^{-\alpha}\big) \mbox{ for all } s \geq s_0\Big)=1
$$
for every $\alpha \in (0,1/2)$. This implies the statement of Theorem \ref{Theorem4}.
\end{proof}

  \section{Remarks and Conclusions} \label{conc}
In this paper, using as an example the SDE (\ref{SDE}), we have identified
and investigated the stochastic equivalent of a repelling river. Several
open questions remain, e.g.~generalizations to other SDEs possibly driven by more general noise and other drift functions.  One open question (reminiscent of the law of the iterated logarithm for Brownian motion) is to find  decreasing positive deterministic functions $l_1$ and $l_2$ such that
$\limsup_{s \to \infty}\frac{\RR(s)-s}{l_1(s)}=1$ and  $\liminf_{s \to \infty}\frac{\RR(s)-s}{l_2(s)}=-1$ almost surely. If such functions exist then Remark \ref{inft} shows that they decrease (slightly) more slowly than
$s^{-1/2}$. Using Corollary \ref{Rbar} it should not be too hard to find such functions by investigating the question for $\overline{\scr{R}}$ instead of $\scr{R}$. 

In the deterministic context, rivers can be
located in the absence of an exact solution by asymptotics
\cite{Blais} or by using the Wa\.zewski principle \cite{Srz,Waz}. We
leave it to future work to determine whether similar methods could be
tailored to work in the stochastic case.

\begin{appendices}

\section{Appendix}
We state some well-known results about scalar diffusion processes which we use in this paper many times. The proofs can be found in \cite[Section 5.5.5]{KS}. Statement c) is often referred to as {\em Feller's test of explosion}. We only consider the additive noise case.
\begin{lem}\label{1ddiff}
  Consider the scalar SDE
  \begin{equation}\label{solu}
\d Z(t) = b(Z(t))\, \d t + \sigma \,\d W(t), 
\end{equation}
where $\sigma>0$, $b:\R \to \R$ is locally Lipschitz continuous and $W$ is one-dimensional Brownian motion.
 
Fix some $c \in \R$ and define the {\em scale function}
$$
p(x):=\int_c^x\exp\Big\{ -\frac 2{\sigma^2}\int_c^\xi b(\zeta)\,\d \zeta\Big\}\,\d \xi,\;x \in \R.
$$
Let $-\infty\leq l < c<r \leq \infty$, let $Z$ be  the solution of \eqref{solu} with initial condition $x \in (l,r)$ and define  the exit time from the interval $(l,r)$ by $S:=\inf\{t \geq 0: Z(t) \notin (l,r)\}$. We define $p(r):=\lim_{y \uparrow r}p(y)$ if $r=\infty$ and similarly for $p(l)$. 
\begin{itemize}
\item [a)] If $p(l)>-\infty$ and $p(r)<\infty$ then $Z$ satisfies 
\begin{equation}\label{ratio}
\P\Big(\lim_{t \to S}Z(t)=l    \Big)=1-\P\Big(\lim_{t \to S}Z(t)=r    \Big)=\frac{p(r)-p(x)}{p(r)-p(l)}.
\end{equation}
(Note that the scale function depends on $c$ but the right hand side of \eqref{ratio} doesn't.)

\item[b)] If $p(l)>-\infty$ and $p(r)<\infty$, then
$$
\E\big[S\big]= -\int_l^x\big(p(x)-p(y)\big)\,m(\d y)+  \frac{p(x)-p(l)}{p(r)-p(l)}\int_l^r\big(p(r)-p(y)\big)\,m(\d y),
$$
where
$$
m(\d y):=\frac 2{p'(y)\sigma^2}\d y,\;y\in \R
$$
is the {\em speed measure} of $Z$.
\item[c)]
  Let
  $$
v(x):=\int_c^xp'(y)\int_c^y\frac 2{p'(z)\sigma^2}\,\d z\,\d y=\int_c^x\big(p(x)-p(y)\big)\,m(\d y).
$$
Then $\P (S=\infty)=1$ iff $v(l)=v(r)=\infty$. If $p(l)=-\infty$ and $v(r)<\infty$, then $S<\infty$ and $\lim_{t \uparrow S}Z(t)=\infty$ almost surely.  
  \end{itemize}
\end{lem}

\begin{proof} \indent

  \begin{itemize}
    \item[a)] \cite[Proposition 5.5.22d)]{KS}
    \item[b)] \cite[(5.5.55),(5.5.59)]{KS}
      \item[c)] \cite[Theorem 5.5.29, Proposition 5.5.32]{KS}
      \end{itemize}
\end{proof}

The following result is \cite[Lemma 2.2]{ESV} in which we replaced the constant $a_{p,\mu}$ by its numerical value given in the proof. 
\begin{lem}\label{ESVlemma}
  Let $p\in (2,\infty)$ and let $W(t),\,t \geq 0$ be a standard Brownian motion on a filtered probability space $(\Omega,\F,(\F_t)_{t \geq 0},\P)$. Further, let $\eta:[0,\infty)\times \Omega \to \R$
  be a progressively measurable process such that, for each $T \in (0,\infty)$,
  $$
\E \bigg[\int_0^T|\eta(s,\omega)|^p\,\d s\bigg]<\infty.
$$
For $\mu>0$ let $v_\mu(t), t\geq 0$ be the unique solution of the equation
\begin{equation*}
\left\{
  \begin{array}{rl}
    dv(t)&=-\mu v(t)\,\d t + \eta(t,\omega) \,\d W(t), \quad t\geq 0\\
    v(0)&=0. 
  \end{array}
  \right.
\end{equation*}
Then, for each $T \in (0,\infty)$,
$$
\E \bigg[ \sup_{0 \leq t \leq T}|v_\mu(t)|^p\bigg] \leq a_{p,\mu}\cdot\E \bigg[ \int_0^T|\eta(s,\omega)|^p\,\d s\bigg],
$$
where, for some constant $\gamma_p\in (0,\infty)$,
$$
a_{p,\mu}=\gamma_p\mu^{1-\frac p2}.
$$
\end{lem}

\end{appendices}


\bmhead{Acknowledgements}

Michael Scheutzow acknowledges financial
support from the London Mathematical Society, grant No. 22203.

\bibliography{SG}

\end{document}